\documentclass[preprint]{elsarticle}
\usepackage[utf8]{inputenc}

\usepackage{lineno,hyperref}
\modulolinenumbers[1]

\usepackage{latexsym,amssymb,amsfonts,amsmath}
\usepackage{amsthm}
\usepackage{mathtools}
\usepackage{verbatim}
\usepackage{paralist}
\usepackage[all,color]{xy}
\usepackage{pdfpages}
\usepackage{setspace}
\usepackage{graphicx}
\usepackage{caption}
\large

\theoremstyle{plain}
\newtheorem{theorem}{Theorem}
\newtheorem{proposition}[theorem]{Proposition}
\newtheorem{lemma}[theorem]{Lemma}

\theoremstyle{definition}
\newtheorem{definition}[theorem]{Definition}

\newtheorem{remark}[theorem]{Remark}
\newtheorem{notation}[theorem]{Notation}
\newtheorem{claim}[theorem]{Claim}

\DeclareMathOperator{\lcm}{lcm}

\makeatletter
\def\ps@pprintTitle{%
   \let\@oddhead\@empty
   \let\@evenhead\@empty
   \let\@oddfoot\@empty
   \let\@evenfoot\@oddfoot
}
\makeatother

\begin{document}

\begin{frontmatter}
\title{Vertex Alternating-Pancyclism in 2-Edge-Colored Graphs\tnoteref{t1}}
\tnotetext[t1]{This research was supported by grants UNAM-DGAPA-PAPIIT IN102320.}

\author[imunam]{Narda Cordero-Michel\corref{cor1}}
\ead{narda@matem.unam.mx}
\author[imunam]{Hortensia Galeana-S\'anchez}
\ead{hgaleana@matem.unam.mx}

\cortext[cor1]{Corresponding author}
\address[imunam]{Instituto de Matem\'aticas, Universidad Nacional Aut\'onoma de M\'exico, Ciudad Universitaria, CDMX, 04510, M\'exico}


\begin{abstract}
An \emph{alternating cycle} in a 2-two-edge-colored graph is a cycle such that any two consecutive edges have different colors.  Let $G_1$, \ldots, $G_k$ be a collection of pairwise vertex disjoint 2-edge-colored graphs. The \emph{colored generalized sum} of   $G_1$, \ldots, $G_k$, denoted by $ \oplus_{i=1}^k G_i$, is the set of all 2-edge-colored graphs $G$ such that: (i) $V(G)=\bigcup_{i=1}^k V(G_i)$, (ii) $G\langle V(G_i)\rangle\cong G_i$ for $i=1,\ldots, k$ as edge-colored graphs where $G\langle V(G_i)\rangle$ has the same coloring as $G_i$ and (iii) between each pair of vertices in different summands of $G$ there is exactly one edge, with an arbitrary but fixed color. 
A graph $G$ in $\oplus_{i=1}^k G_i$ will be called a \emph{colored generalized sum} (c.g.s.) and we will say that $e\in E(G)$ is an  \emph{exterior edge} iff $e\in E(G)\setminus \left(\bigcup_{i=1}^k E(G_i)\right)$. The set of exterior edges will be denoted by $E_\oplus$.
A colored graph $G$ is said to be a \emph{vertex alternating-pancyclic graph}, whenever for each vertex $v$ in $G$, and for each $l\in\{3,\ldots, |V(G)|\}$, there exists in $G$ an alternating cycle of length $l$ passing through $v$.

The topics of pancyclism and vertex-pancyclism are deeply and widely studied by several authors. The existence of alternating cycles in 2-edge-colored graphs has been studied because of its many applications. 
In this paper, we give sufficient conditions for a graph $G\in \oplus_{i=1}^k G_i$ to be a vertex alternating-pancyclic graph.
\end{abstract}


\end{frontmatter}


\section{Introduction}
\label{sec:intro}

Let $G$ be an edge-colored multigraph. 
An \emph{alternating walk} in $G$, is a walk such that any two consecutive edges have different colors. 

Several problems have been modeled by edge-colored multigraphs, the study of applications of alternating walks seems to have started in \cite{Petersen1891193}, according to \cite{Bang-Jensen199739}, and ever since it has crossed diverse fields, such as genetics \cite{Doringer198795,Dorninger1994159,Dorninger1987321,Pevzner199577}, transportation and connectivity problems \cite{Gourves20101404,Wirth200173}, social sciences \cite{Chow199449} and graph models for conflict resolutions \cite{Xu2010318,Xu2009470,Xu2009353}, as pointed out in \cite{Contreras-Balbuena201755}.

The Hamiltonian alternating path and cycle problems are $\mathcal{NP}$-complete even for two colors, it was proved in \cite{Gorbenko2012204}, and so the problem of deciding if a given graph is alternating-pancyclic is as difficult as those two problems.

In \cite{Das1982105}, Das characterized 2-edge-colored complete bipartite multigraphs which are vertex alternating-pancyclic and, in \cite{Bang-Jensen199739}, Bang-Jensen and Gutin characterized 2-edge-colored complete multigraphs which are vertex alternating-pancyclic.

In other publications, such as \cite{Fujita20111391} and \cite{Wang20094349}, authors studied the existence of alternating cycles of certain lengths in terms of vertex degrees.  \\

In this paper we study a broad class of 2-edge-colored graphs, namely the colored generalized sum of 2-edge-colored graphs, we give sufficient conditions for a graph in the c.g.s. of $k$ Hamiltonian alternating graphs to be a vertex alternating-pancyclic graph:

A graph $G\in \oplus_{i=1}^k G_i$, where $G_1$, $G_2$, \ldots, $G_k$ is a collection of $k\geq 2$ vertex disjoint graphs with Hamiltonian alternating cycles, $C_1$, $C_2$, \ldots, $C_k$, respectively, is vertex alternating-pancyclic if $G$ contains no good cycle and, for each pair of different indices $i,j\in [1,k]$, in $C_i$ there is a non-singular vertex with respect to $C_{j}$.\footnote{The definition of good cycle is in Section \ref{sec:main results} and the definition of non-singular vertex is in Section \ref{sec:definitions}.}

It should be noted that the proofs in this paper carry on an implicit algorithm to construct the alternating cycles.

\section{Definitions}
\label{sec:definitions}
In this paper $G=(V(G),E(G))$ will denote a simple graph. 
A \emph{$k$-edge-coloring} of $G$ is a function $c$ from the edge set, $E(G)$, to a set of $k$ colors, $ \{1,2,\ldots,k\}$. A graph $G$ provided with a $k$-edge-coloring is a \emph{$k$-colored-graph}. 

A path or a cycle in  $G$  will be called an \emph{alternating path} or an \emph{alternating cycle} whenever two consecutive edges have different colors. 
An alternating cycle containing each vertex of the graph is a \emph{Hamiltonian alternating cycle} and a graph containing a Hamiltonian alternating cycle will be called a \emph{Hamiltonian alternating graph}.
A 2-edge-colored graph $G$ of order $2n$  is \emph{vertex alternating-pancyclic} iff, for each vertex $v\in V(G)$ and each $k\in \{2,\ldots,n\}$, $G$ contains an alternating cycle of length $2k$ passing through $v$. 

For further details we refer the reader to \cite{Bang-Jensen2009} pages 608-610.



\begin{remark}
Clearly the c.g.s. of two vertex disjoint graphs if well defined.
Let $G_1$, $G_2$, $G_3$ be three vertex disjoint 2-edge-colored graphs. It is easy to see that the sets $(G_1\oplus G_2)\oplus G_3$ defined as $\bigcup_{G\in G_1\oplus G_2} G\oplus G_3 $ and $G_1 \oplus (G_2\oplus G_3)$ defined as $\bigcup_{G'\in G_2\oplus G_3} G_1\oplus G'$ are equal, thus $\oplus_{i=1}^3 G_i=(G_1\oplus G_2)\oplus G_3 =G_1\oplus(G_2\oplus G_3)$ is well defined. By means of an inductive process it is easy to see that  the c.g.s. of $k$ vertex disjoint 2-edge-colored graphs is well defined and  associative.
\end{remark}

\begin{remark}
\label{remark subsuma}
Let $G_1$, $G_2$, \ldots, $G_k$ be a collection of pairwise vertex disjoint 2-edge-colored graphs; $G\in\oplus_{i=1}^k G_i$; and $J\subset [1,k]$. The induced subgraph of $G$ by $\bigcup_{j\in J} V(G_j)$, $H=G\langle \bigcup_{j\in J} V(G_j) \rangle$, belongs to the c.g.s. of $\{G_j\}_{j\in J}$.
\end{remark}

\begin{notation}
Let $k_1$ and $k_2$ be two positive integers, where $k_1\leq k_2$. We will denote by $[k_1,k_2]$ the set of integers $\{k_1, k_1+1, \ldots, k_2\}$.
\end{notation}

\begin{notation}
\label{neighbor in alternating cycle}
Let $C=x_0 x_1 \cdots x_{2n-1} x_0$ be an alternating cycle. For each $v\in V(C)$, we will denote by $v^r$ (resp. $v^b$) the vertex in $C$ such that $vv^r\in E(C)$  is red (resp. $vv^b\in E(C)$ is blue). 
Notice that if $v=x_i$ then $\{x_{i-1},x_{i+1}\}=\{v^r,v^b\}$.

If more that one alternating cycle contains $v$, we will write $v^{r}_C$ (resp. $v^{b}_C$).
\end{notation}

\begin{definition}
Let $G$ be a 2-edge-colored graph and let $C_1=x_0x_1\cdots x_{2n-1}x_0$ and $C_2=y_0y_1\cdots y_{2m-1}y_0$ be two vertex disjoint alternating cycles. Let $vw$ be an edge with $v\in V(C_1)$ and $w \in V(C_2)$. If $c(vw)=\text{red}$ (resp. $c(vw)=\text{blue}$) we will say that $vw$, $v^rw^r$ (resp. $vw$, $v^bw^b$) is a \emph{good pair of edges} whenever $c(v^rw^r)=\text{red}$ (resp. $c(v^bw^b)=\text{blue}$).

Whenever there is a good pair of edges between two vertex disjoint alternating cycles $C_1$ and $C_2$, we simply say that there is a good pair.
\end{definition}

\begin{remark}
\label{remark:monochromatic cycle}
Notice that $vv^rw^rwv$ (resp. $vv^bw^bwv$) is a monochromatic 4-cycle whenever  $vw$, $v^rw^r$ (resp. $vw$, $v^bw^b$) is a good pair.
\end{remark}

\begin{figure}
\begin{center}   
\includegraphics[scale=0.9]{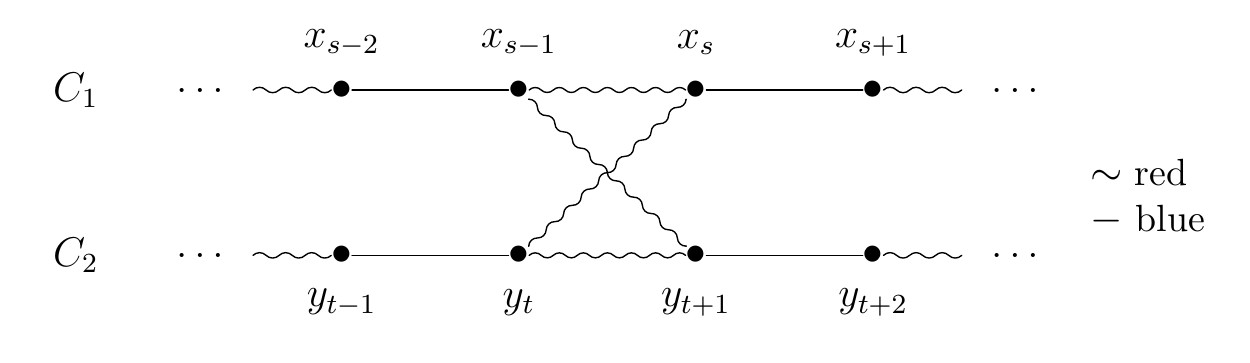}
\caption{A good pair of edges.}
\label{fig:good pair}
\end{center} 
\end{figure}

\begin{remark}
\label{remark par bueno}
Let $G$ be a 2-edge-colored graph and let $C_1=x_0x_1\cdots x_{2n-1}x_0$ and $C_2=y_0y_1\cdots y_{2m-1}y_0$ be two vertex disjoint alternating cycles. A pair of edges  $x_sy_t$, $x_{s'}y_{t'}$ with $s\in [0, 2n-1]$, $s'\in \{s-1, s+1\}$, $t\in [0, 2m-1]$, $t'\in \{t-1,t+1\}$ where all the subscripts are taken modulo $2n$ and $2m$, respectively,  is  a good pair  whenever $x_sx_{s'}y_{t'}y_tx_s$ is a monochromatic 4-cycle (Figure \ref{fig:good pair}). This is consequence of the definition of a good pair and Notation \ref{neighbor in alternating cycle}.
\end{remark}

In the study of alternating cycles, the more general case is the one with two colors and so we will work with 2-edge-colored graphs. 
In what follows: any graph $G$ will denote a 2-edge-colored graph and $c: E(G)\to \{\text{red, blue}\}$ will denote its edge coloring and we will simply say a graph instead of a 2-edge-colored graph; a 2-edge-colored cycle $C$ which is properly colored will simply be called an alternating cycle. In our figures curly lines will represent red edges while straight lines will represent blue edges, sometimes we will use dotted lines to represent edges that we ignore to construct a cycle, we will use double-dotted lines for red edges and dotted lines for blue edges.\\

From now on the subscripts for vertices in $C_1=x_0x_1\cdots x_{2n-1}x_0$ will be taken modulo $2n$ and for vertices in $C_2=y_0y_1\cdots y_{2m-1}y_0$ will be taken modulo $2m$.

\section{Preliminary results}
\label{sec:preliminary results}

In this section we will see a series of results that describe the behavior of exterior edges in a c.g.s. of two 2-edge-colored graphs.

\begin{proposition}
\label{propo merging cycles with a good pair of edges}
Let $C_1$ and $C_2$ be two disjoint alternating cycles in a graph $G$. If there is a good pair of edges between them, then there is an alternating cycle with the vertex set $V(C_1)\cup V(C_2)$.
\end{proposition}
\begin{proof}
Let $C_1 = x_0 x_1 \cdots  x_{2n-1} x_0$ and $C_2 = y_0y_1\cdots y_{2m-1}y_0$ be two vertex disjoint alternating cycles in $G$. 

Let $x_sy_t$ and $x_{s'}y_{t'}$ be a good pair, as in Remark \ref{remark par bueno}.
Without loss of generality suppose that $x_sx_{s'}y_{t'}y_tx_s$ is a red monochromatic 4-cycle. 
 
We will prove the case where $x_{s'}=x_{s-1} $ and $y_{t'}=y_{t+1}$,  \textit{i.e.}, $x_sx_{s-1}$, $x_{s-1}y_{t+1}$, $y_{t+1}y_t$ and $y_tx_s$ are red. The other three cases can be proved in a similar way.  
Since $C_1$ and $C_2$ are alternating cycles, we have that the edges $x_{s-2}x_{s-1}$, $x_{s}x_{s+1}$, $y_{t-1}y_t$ and $y_{t+1}y_{t+2}$ must be blue. 

Hence, $C=x_sC_1x_{s-1}y_{t+1}C_2y_{t}x_{s}$ is an alternating cycle with vertex set $V(C)=V(C_1)\cup V(C_2)$ (Figure \ref{fig:good pair of edges}). 
\end{proof}

\begin{figure}
\begin{center}   
\includegraphics[scale=0.9]{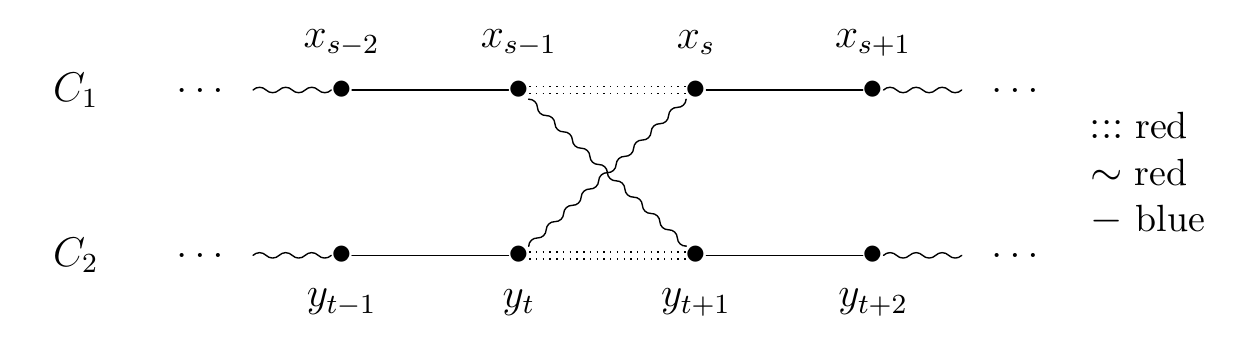}
\caption{A cycle using a good pair of edges.}
\label{fig:good pair of edges}
\end{center} 
\end{figure}

\begin{definition}
\label{definition parallel}

Let $G_1$ and $G_2$ be two vertex disjoint  graphs with Hamiltonian alternating cycles $C_1=x_0x_1\cdots x_{2n-1}x_0$ and $C_2=y_0 y_1 \cdots y_{2m-1} y_0$, respectively; $G\in G_1\oplus G_2$ be a c.g.s. with no good pair. For each exterior edge $u_0v_0$, $u_0\in V(C_1)$ and $v_0\in V(C_2)$, we construct a sequence of edges as follows: (i) If $u_0v_0$ is red (blue), let $u_1\in V(C_1)$ and $v_1\in V(C_2)$ such that, $u_0u_1\in E(C_1)$ is red (blue) and $v_0v_1\in E(C_2)$ is red (blue).
Since $G$ has no good pair, we have that $u_1v_1$ is blue (red). Observe that $u_0v_0$ and $u_1v_1$ have different colors. (ii) Assume that $u_0v_0$, $u_1v_1$, \ldots, $u_iv_i$ have been constructed in such a way that $u_0$, $u_1$, \ldots, $u_i\in V(C_1)$, $v_0$, $v_1$, \ldots, $v_i\in V(C_2)$, $c(u_jv_j)\neq c(u_{j+1}v_{j+1})$, $c(u_ju_{j+1})=c(v_jv_{j+1})$ for each $j$, $0\leq j \leq i-1$, and the sequence $u_0u_1\cdots u_i$ (resp. $v_0v_1\cdots v_i$) is a walk (notice that: when $l(C_1)\neq l(C_2)$ a same vertex can appear more than once in this process) contained in $C_1$ (resp. $C_2$) that moves along $C_1$ (resp. $C_2$)  in the same direction or in the opposite direction (thus, consecutive edges have different colors). (iii) Now, we construct $u_{i+1}v_{i+1}$. If $u_iv_i$ is red (blue), let $u_{i+1}\in V(C_1)$ and $v_{i+1}\in V(C_2)$ such that $u_iu_{i+1}\in E(C_1)$ is red (blue) and $v_iv_{i+1}\in E(C_2)$ is red (blue). Since $G$ has no good pair, we have that $u_{i+1}v_{i+1}$ is blue (red). (iv) The sequence finishes the first time that $u_kv_k=u_0v_0$. That is when $k=\lcm(2n,2m)$. Notice that $u_iv_i$ and $u_{i+1}v_{i+1}$ have different colors (and they are different edges), and thus $c(u_iv_i)=c(u_jv_j)$  if and only if $i\equiv j \pmod{2}$ (Figure \ref{fig:parallel class}). 

The \emph{parallel class} of $u_0v_0$ denoted by $P_{u_0v_0}$ is the set $\{u_iv_i\ |\ 0\leq i\leq k-1\}$ and two exterior edges $e_1, e_2$ will be named \emph{parallel} whenever there exists an exterior edge $zw$ such that $e_1, e_2 \in P_{zw}$.

Since $k=\lcm(2n,2m)$, for each vertex $w\in V(C_i)$, $i\in\{1,2\}$, there exists at least one edge in the parallel class $P_{u_0v_0}$ incident with $w$.
\end{definition}  

\begin{figure}
\begin{center}   
\includegraphics[scale=0.86]{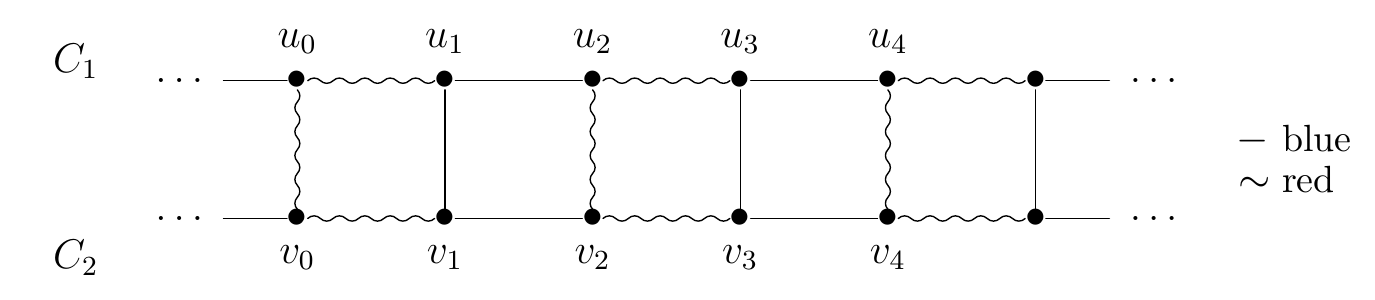}
\caption{The parallel class of $u_0v_0$, $P_{u_0v_0}$.}
\label{fig:parallel class}
\end{center} 
\end{figure}

\begin{remark}
\label{remark parallel class}
As a consequence of Definition \ref{definition parallel} we have that for any $uv\in E_\oplus$, the following assertions hold: (i) $uv\in P_{uv}$;  
(ii) $\left\vert P_{uv} \right\vert=\lcm(2n,2m)=2\lcm(n,m)$; (iii) $P_{uv}$ has $\lcm(n,m)$ red edges and $\lcm(n,m)$ blue edges; (iv) if $P_{uv}\cap P_{wz}\neq \emptyset$, then $P_{uv}= P_{wz}$. 

Observe that, whenever $n\neq m$ it may be more than one edge in $P_{uv}$ incident with $u$ or $v$.
\end{remark}

\begin{remark}
\label{partition of exterior edges}
Notice that, in view of Remark \ref{remark parallel class}, the set of parallel classes, $\mathcal{P}=\{P_{uv}\ |\ uv\in E_\oplus\}$, becomes a partition of the exterior edges.
\end{remark}

\begin{notation}
Let $C$ be an alternating cycle, given an arbitrary but fixed description of its vertices as $C=x_0x_1\cdots x_{2n-1}x_0$, we will say that two vertices $x,y\in V(C)$ are \emph{congruent modulo 2}, whenever their subscripts in $C$ are congruent modulo 2, and we will write $x\equiv y \pmod{2}$ (or by $x\equiv_C y \pmod{2}$, when $x$ and $y$ belong to more than one cycle).
\end{notation}

\begin{notation}
Let $G_1$, $G_2$, \ldots, $G_k$ be a collection of pairwise vertex disjoint 2-edge-colored graphs and take $G$ in $\oplus_{i=1}^k G_i$. For each $v\in V(G)$, we will denote by $d_r(v)$ (resp. $d_b(v)$) the number of red (resp. blue) exterior edges of $G$ incident with $v$.
\end{notation}

\begin{lemma}
\label{lemma no good pair}
Let $C_1=x_0x_1\cdots x_{2n-1}x_0$ and $C_2=y_0y_1\cdots y_{2m-1}y_0$ be two vertex disjoint  alternating cycles and $G$ be a  graph in $C_1 \oplus C_2$ such that $G$ has no good pair. 
The following assertions hold:
(i) There are exactly $2mn$ exterior edges of each color in $G$; (ii) Let $P_{uv}$ be a parallel class, each vertex $w\in V(C_i)$ is incident with  $\frac{\lcm(n,m)}{n}$ edges in $P_{uv}$ for $i=1$;   $\frac{\lcm(n,m)}{m}$ edges in $P_{uv}$ for $i=2$; and all of them are colored alike. Moreover, for vertices $w,x\in V(C_i)$, the edges in $P_{uv}$ incident with $w$ are colored alike as those incident with $x$ if and only if $w\equiv x \pmod{2}$; (iii) For each vertex $w\in V(C_i)$, if $d_r(w)=t$ and $d_b(w)=|V(C_{3-i})|-t $, then $d_r(x)=|V(C_{3-i})|- t=d_b(w)$ and $d_b(x)=t=d_r(w)$ for each $x\in \{w^r,w^b\}$. Furthermore, if $w,x\in V(C_i)$ then: 

\begin{equation*} d_r(x)=\begin{cases} d_r(w)&\text{iff }w\equiv x\pmod{2}\\   |V(C_{3-i})|-d_r(w)&\text{iff }w\not\equiv x\pmod{2}.
\end{cases}
\end{equation*} 
\end{lemma}
\begin{proof}
Let $G$ be as in the hypothesis. \\
(i) Clearly, $G$ has $4mn$ exterior edges and, from  Remark \ref{remark parallel class}, there are the same number of red exterior edges as blue exterior edges. Then there are $2mn$ red exterior edges and $2mn$ blue exterior edges.\\
(ii) Recall that from Remark \ref{remark parallel class} (i) $uv\in P_{uv}$. Let $P_{uv}=\{uv=u_0v_0, u_1v_1$,\ldots, $u_{k-1}v_{k-1}\}$, where $k=\lcm(2n,2m)$, as in Definition \ref{definition parallel}. By Definition \ref{definition parallel}, we have the following assertions: 
(a) $c(w_jz_j)=c(w_0z_0)$ iff $j\equiv 0 \pmod{2}$; (b) $k=\lcm(2n,2m)=2\lcm(n,m)$ (so $k$, $2n$ and $2m$ are even numbers); (c) $u_0=u_j$ iff $j\equiv 0 \pmod{l(C_i)}$, where $u=u_0\in V(C_i)$. 
Hence, the edges in $P_{wz}$ incident with $u=u_0$ are $\frac{\lcm(n,m)}{n}$ if $i=1$, $\frac{\lcm(n,m)}{m}$ if $i=2$; and all of them are colored alike.

Clearly (a) and (b) imply the last assertion in (ii).\\
(iii) By Remark \ref{partition of exterior edges}, $\mathcal{P}=\{P_{uv}\ |\ uv\in E_\oplus\}$ is a partition of $E_\oplus$. From (ii) in this lemma, for each $P_{uv}$ the number of edges incident with $w$, $w^r$ and $w^b$ is $\frac{2\lcm(n,m)}{l(C_i)}$ as $w, w^r,w^b \in V(C_i)$, for some  $i\in\{1,2\}$, and $w^r\equiv w^b \pmod{2}$. Hence, the edges in a parallel class incident with $w^r$ and those incident with $w^b$ are colored alike, and those incident with $w$ have a different color, because each parallel class is a sequence of edges with alternating colors and if two edges are consecutive in this sequence, their ends are consecutive vertices in the cycles $C_i$ and $C_{3-i}$.
Therefore, if $d_r(w)=t$ then $d_b(w)=|V(C_{3-i})|-t$, $d_r(x)=d_b(w)$ and $d_b(x)=d_r(w)$ for $x\in \{w^r,w^b\}$.

The last assertion in (iii) follows directly from above, because $C_i$ is alternating, for each $i\in\{1,2\}$.  
\end{proof}

\begin{definition}
Let $\mathcal{F}=\{C_1, C_2\}$ be an alternating cycle factor in a  graph $G$. A vertex $v\in V(C_i)$ is \emph{red-singular} (\emph{blue-singular}) with respect to $C_{3-i}$ if $\{vu\ |\ u\in V(C_{3-i})\}$ is not empty and all the edges in $\{vu\ |\ u\in V(C_{3-i})\}$ are red (blue); $v$ is \emph{singular} with respect to $C_{3-i}$, if it is either red-singular or blue-singular with respect to $C_{3-i}$.

Let $\mathcal{F'}=\{C, H\}$ be a factor in a  graph $G$, where $C$ is an alternating cycle and $H$ is a subgraph. A vertex $v\in V(C)$ is \emph{red-singular} (\emph{blue-singular}) with respect to $H$, if $\{vu\ |\ u\in V(C_{3-i})\}$ is not empty and all the edges in $\{vu\ |\ u\in V(H)\}$ are red (blue); $v$ is \emph{singular} with respect to $H$ if it is either red-singular or blue-singular with respect to $H$.

\end{definition}

Now we will prove a lemma that is useful in the proofs of Propositions \ref{proposition non-singular vertices and no good pair then almost  vertex pancyclic} and \ref{proposition non-singular vertices and no good pair then almost  vertex pancyclic 2} which give sufficient conditions for a graph in the c.g.s. of two Hamiltonian alternating 2-edge-colored graphs to be a vertex alternating-pancyclic graph.

\begin{lemma}
\label{lemma non-singular vertices, two exterior edges of different colors}
Let $G_1$ and $G_2$ be two vertex disjoint graphs with Hamiltonian alternating cycles, $C_1=x_0x_1\cdots x_{2n-1}x_0$ and $C_2=y_0y_1\cdots y_{2m-1}y_0$, respectively, and $G\in G_1\oplus G_2$. If there is no good pair in $G$, and for each $i\in \{1,2\}$, there is a non-singular vertex with respect to $C_{3-i}$ in $C_i$. Then there exist $r\in [0,2m-1]$ and $s\in[0,2n-1]$ such that $x_0y_ry_{r+1}y_{r+2}x_0$ and $y_0x_sx_{s+1}x_{s+2}y_0$ are alternating 4-cycles. (Figure \ref{fig alternating 4-cycle}).
\end{lemma}
\begin{proof}
Since there are non-singular vertices $u\in V(C_1)$ with respect to $C_2$ and $v\in V(C_2)$ with respect to $C_1$. 
Then there are red and blue exterior edges incident with $u$ (resp. $v$). Hence, $d_r(u)\geq 1$, $d_b(u)\geq 1$, $d_r(v)\geq 1$, $d_b(v)\geq 1$. Since there is no good pair, by Lemma \ref{lemma no good pair} (ii), $d_r(w)\geq 1$, $d_b(w)\geq 1$ for each $w\in V(C_1)\cup V(C_2)=V(G)$, and thus $C_i$ has no singular vertex with respect to $C_{3-i}$, for each $i\{1,2\}$.

\begin{figure}[!h]
\begin{center}
\includegraphics[scale=0.9]{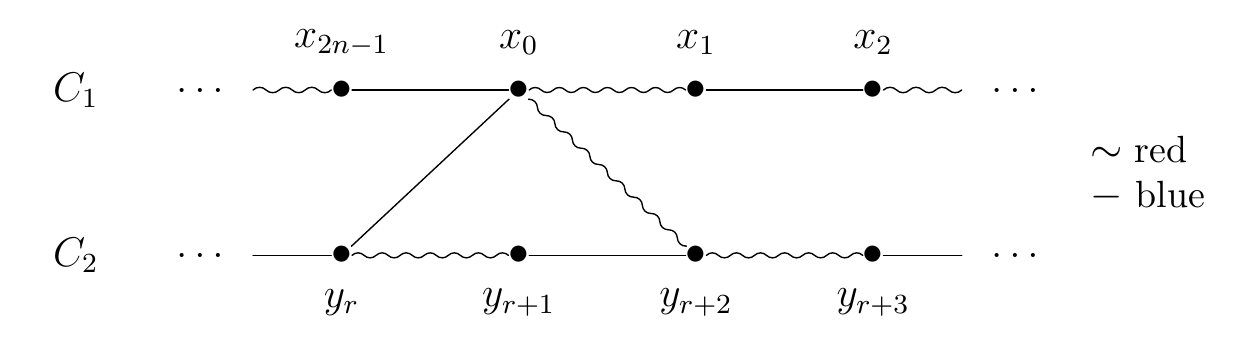}
\caption{An alternating 4-cycle $x_0y_{r'}y_{r}y_{r''}x_0$.} 
\label{fig alternating 4-cycle}
\end{center}
\end{figure}

Proceeding by contradiction, suppose w.l.o.g. that there is no such $r\in[0,2m-1]$ (in the case in which there is no such $s\in [0,2n-1]$ the contradiction arises in a very similar way).

Take $x_0y_h$ an exterior blue edge. Then $y_hy_{h+1}$ or $y_hy_{h-1}$ is red, w.l.o.g. we assume that $y_hy_{h+1}$is red, for a fixed $h\in [0,2m-1]$. 
Our assumption implies that $x_0y_{h+2}$ is blue. Now, take $x_0y_{h+2}$, arguing as before, we get that $x_0y_{h+4}$ is blue. Continuing this way, we conclude that  $x_0y_{h+2j}$ is blue for each $j\in [0,m-1]$. Since $C_i$ has no singular vertices with respect to $C_{3-i}$, there is a  $x_0y_{h+2t+1}$ red exterior edge, for some $t\in[0,m-1]$. Since $y_{h+2t+1}y_{h+2t+2}$ is blue, arguing as above we get that $x_0y_{h+2t+1}$ is red for each $t\in[0,m-1]$. 

Recall that $G$ has no good pair. Therefore, by Lemma \ref{lemma no good pair} (ii) all edges parallel to $x_0y_{h+2j}$  having $y_h$ as an end are blue and those parallel to $x_0y_{h+2j+1}$ having $y_h$ as an end are blue, for each $j\in [0,m-1]$. So we conclude that $y_h$ is a blue-singular vertex, a contradiction. 

\end{proof}

In the next two propositions, for a graph $G$ in the c.g.s. of tho Hamiltonian alternating 2-edge-colored graphs, $G_1$ and $G_2$, which contains no good pair, we will construct, for each $v\in V(G)$,  alternating cycles of each even length passing through $v$.

\begin{notation}
Let $k$ be a positive integer and $A$ be a set of non-negative integers. We will denote by $kA$ to the set of products $\{ka\ |\ a \in A\}$.
\end{notation}

\begin{proposition}
\label{proposition non-singular vertices and no good pair then almost  vertex pancyclic}
Let $G_1$ and $G_2$ be two vertex disjoint graphs with Hamiltonian alternating cycles, $C_1=x_0x_1\cdots x_{2n-1}x_0$ and $C_2=y_0y_1\cdots y_{2m-1}y_0$, respectively, and $G\in G_1\oplus G_2$. If there is no good pair in $G$, and for each $i\in \{1,2\}$, in $C_i$  there is a non-singular vertex with respect to $C_{3-i}$. Then for each vertex $v\in V(G)$ there is an alternating cycle  of each even length in $[4,4m_1]\cup [2M_2, 2n+2m]$, where $m_1=\min\{n,m\}$ and $M_2=\max\{n,m\}$, passing through $v$.
\end{proposition}
\begin{proof}
Suppose w.l.o.g. that $n\geq m$.

By Lemma \ref{lemma non-singular vertices, two exterior edges of different colors}, there exists $r\in [0,2m-1]$ such that $x_0y_{r}y_{r+1}y_{r+2}x_0$ is an alternating 4-cycle.
 
Assume w.l.o.g. that $x_0y_r$ is blue, then $y_ry_{r+1}$ and $x_0y_{r+2}$ are red and $y_{r+1}y_{r+2}$ is blue. 
As $C_2$ is an alternating cycle $y_{r+2j}y_{r+2j+1}$ is red for each $j\in[0,m-1]$ and $y_{r+2j+1}y_{r+2j+2}$ is blue for each $j\in[0,m-1]$. 

We may assume $x_0x_1$ is red. Since $G$ has no good pair, the next assertions follow directly from Definition \ref{definition parallel} and Lemma \ref{lemma no good pair} (ii):

\begin{itemize}
\itemsep-0.2em 
\item $x_{2n-2j}y_{r-2j}$ is an exterior blue edge for each $j \in[0,l-1]$, where $l=\lcm(n,m)$, (\emph{i.e.} $x_{2j}y_{r+2j}$ is an exterior blue edge for each $j \in[0,l-1]$).
\item $x_{2n-2j-1}y_{r-2j-1}$ is an exterior red edge for each $j \in[0,l-1]$, where $l=\lcm(n,m)$, (\emph{i.e.} $x_{2j+1}y_{r+2j+1}$ is an exterior red edge for each $j \in[0,l-1]$).
\item $x_0y_{r+2}$ is red.
\item $x_{2j}y_{r+2j+2}$ is red for each $j\in[0,l-1]$.
\item $x_{2j+1}y_{r+2j+3}$ is blue for each $j\in[0,l-1]$.\\
(Figure \ref{fig plenty alternating 4-cycles}.)
\end{itemize} 
\begin{figure}[!h]
\begin{center}
\includegraphics[width=\textwidth]{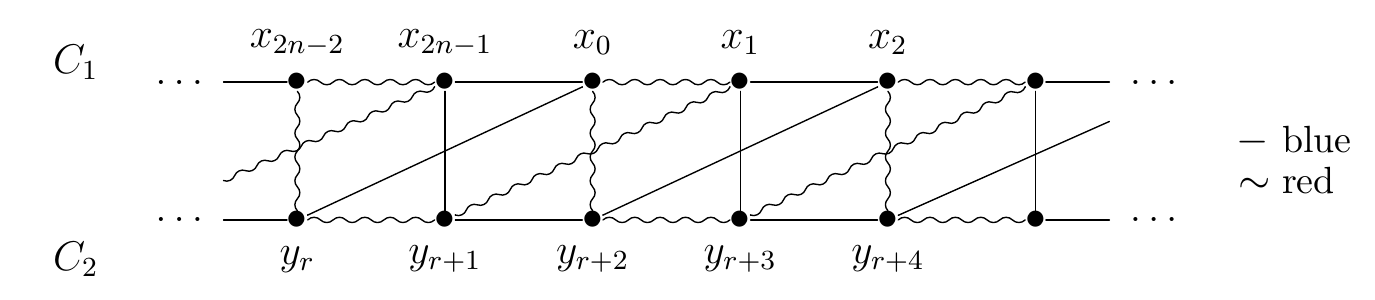}
\caption{Parallel classes of $x_0y_r$ and $x_0y_{r+2}$.} 
\label{fig plenty alternating 4-cycles}
\end{center} 
\end{figure}

\begin{sloppypar}
Let $t\in[0,2n-1]$ and $h\in[0,2m-3]$. Consider the cycle $\alpha_{h}^t=x_ty_{r+t}y_{r+t+1}\cdots y_{r+t+h+2}x_{t+h}x_{t+h-1}\cdots x_t$. 
\end{sloppypar}

\begin{claim}
The cycle $\alpha_{h}^t$ is alternating for each $t\in[0,2n-1]$ and $h\in[0,2m-3]$.

We have four possibilities for the parities of $t$ and $h$.
\begin{sloppypar}
\begin{enumerate}
\itemsep-0.2em 
\item[{Case 1}]: $t$ is even.
Then $x_ty_{r+t}$ is blue, $x_ty_{r+t+2}$ is red, $y_{r+t}y_{r+t+1}$ is red.
\item[{Case 1.1}]: $h$ is even. 
Then $y_{r+t}y_{r+t+1}\cdots y_{r+t+h}$ is an alternating path whose first edge is red and its last edge is blue, 
$x_{t+h}y_{r+h+2}$ is red and  
$x_{t+h}x_{t+h-1}\cdots x_t$ is an alternating path whose first edge is blue and its last edge is red.
\item[{Case 1.2}]: $h$ is odd. 
Then $y_{r+t}y_{r+t+1}\cdots y_{r+t+h}$ is an alternating path whose first and last edges are red, 
$x_{t+h}y_{r+h+2}$ is blue and 
$x_{t+h}x_{t+h-1}\cdots x_t$ is an alternating path whose first and last edges are red.
\item[{Case 2}]: $t$ is odd. 
Then $x_ty_{r+t}$ is red, $x_ty_{r+t+2}$ is blue, $y_{r+t}y_{r+t+1}$ is blue.
\item[{Case 2.1}]:  $h$ is even.
Then $y_{r+t}y_{r+t+1}\cdots y_{r+t+h}$ is an alternating path whose first edge is blue and its last edge is red, 
$x_{t+h}y_{r+h+2}$ is blue and
$x_{t+h}x_{t+h-1}\cdots x_t$ is an alternating path whose first edge is red and its last edge is blue.
\item[{Case 2.2}]: $h$ is odd.
Then $y_{r+t}y_{r+t+1}\cdots y_{r+t+h}$ is an alternating path whose first and last edges are blue, 
$x_{t+h}y_{r+h+2}$ is red and  
$x_{t+h}x_{t+h-1}\cdots x_t$ is an alternating path whose first and last edges are blue.
\end{enumerate}
 \end{sloppypar}
Hence, the cycle $\alpha_{h}^t$ is alternating for each $t\in[0,2n-1]$ and each $h\in[0, 2m-3]$.
\end{claim}
The cycle $\alpha_{h}^t$ has length $l(\alpha_{h}^t)=1+(h+2)+1+h=2h+4$ for each $h\in [0,2m-3]$. Therefore, there exists an alternating cycle of each even length in $[4,\ 4m]$ passing through $x_t$ and $y_{r+t}$, for each $t\in[0,2n-1]$.

Let $t\in [0,2n-1]$. The cycle $\beta^t=x_ty_{r+t+2}x_{t+2}x_{t+3}\cdots x_t$ is alternating since it is obtained from $C_1$ by exchanging the alternating subpath $x_tx_{t+1}x_{t+2}$ of $C_1$ by the path $x_ty_{r+t+2}x_{t+2}$; notice that $c(x_tx_{t+1})=c(x_ty_{r+t+2})$ and $c(x_{t+1}x_{t+2})=c(y_{r+t+2}x_{t+2})$. 

Therefore, for each $t\in [0,2n-1]$, $\beta^t$ is an alternating cycle of length $2n$ containing $x_t$ and $y_{r+t+2}$. Hence, for each vertex $w\in V(G)$ there exists an alternating cycle of length $2n$ passing through $w$.\\
 
\begin{sloppypar}
Finally, let $t\in 2[0,n-1]$ and $h\in[0,m-1]$. Now consider the cycle 
$\gamma_{h}^t=x_ty_{r+t+2}y_{r+t+1}x_{t+1} x_{t+2}y_{r+t+4}y_{r+t+3}x_{t+3}\cdots x_{t+2h}y_{r+t+2h+2}y_{r+t+2h+1}x_{t+2h+1}$ $x_{t+2h+2} C_1 x_t$ 
\end{sloppypar}

\begin{claim}
The cycle $\gamma_{h}^t$ is alternating for each $t\in 2[0,n-1]$ and $h\in[0,m-1]$.
\begin{sloppypar}
As $t$ is even $x_{t+2j}y_{r+t+2j+2}$ is red, $y_{r+t+2j+2} y_{r+t+2j+1}$ is blue, $y_{r+t+2j+1}x_{t+2j+1}$ is red and $x_{t+2j+1}x_{t+2j+2}$ is blue for each $j \in [0,m-1]$ and $x_{t-1}x_t$ is blue. Therefore, $\gamma_h^t$ is alternating  for each $t\in 2[0,n-1]$ and $h\in[0,m-1]$.
(Figure \ref{fig big alternating cycle}.)
\end{sloppypar}
\end{claim}
\begin{figure}[!h]
\begin{center}
\includegraphics[width=\textwidth]{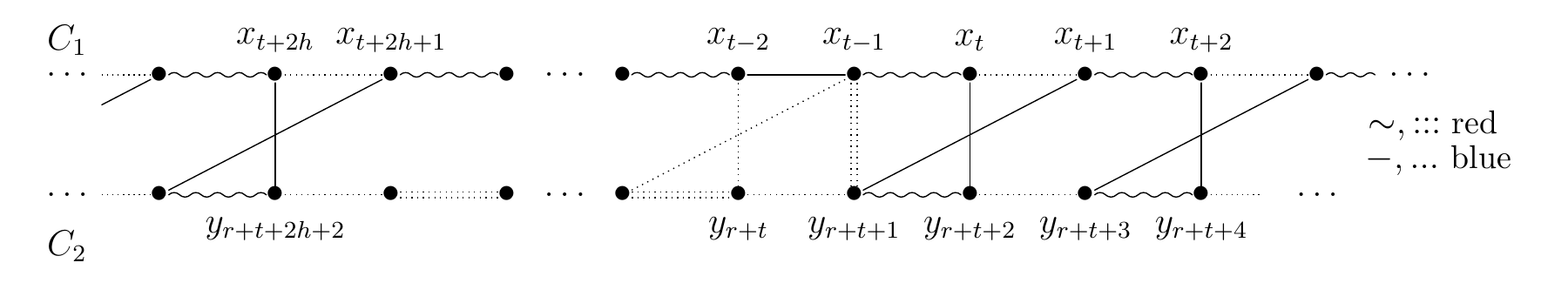}
\caption{An alternating $\gamma_h^t$ cycle in $G$.} 
\label{fig big alternating cycle}
\end{center} 
\end{figure}
 
Now, the cycle $\gamma_h^t$ has length $l(\gamma_h^t)=2n+2h+2$ for each $h\in [0,m-1]$ and 
thus we have alternating cycles of each even length in $[2n+2,\ 2n+2m]$  passing through $x_t$, $y_{r+t+1}$ and $y_{r+t+s}$ for each $t\in[0,2n-1]$. Since $n\geq m$ we have that for each $w\in V(G)$ there exist an alternating cycle of each even length in $[2n+2,\ 2n+2m]$ passing through $w$.

\end{proof}

\begin{proposition}
\label{proposition non-singular vertices and no good pair then almost  vertex pancyclic 2}
Let $G_1$ and $G_2$ be two vertex disjoint graphs with Hamiltonian alternating cycles, $C_1=x_0x_1\cdots x_{2n-1}x_0$ and $C_2=y_0y_1\cdots y_{2m-1}y_0$, respectively, and $G\in G_1\oplus G_2$. If there is no good pair in $G$, and for each $i\in \{1,2\}$, in $C_i$  there is a non-singular vertex with respect to $C_{3-i}$. Then for each vertex $v\in V(G)$ there is an alternating cycle  of each even length in $[4m_1,\ 2M_2]$, where $m_1=\min\{n,\ m\}$ and $M_2=\max\{n,\ m\}$, passing through $v$.
\end{proposition}
\begin{proof}
Suppose w.l.o.g. that $n\geq m$.

By Lemma \ref{lemma non-singular vertices, two exterior edges of different colors}, there exists $s\in [0,2n-1]$ such that $y_0x_sx_{s+1}x_{s+2}y_0$ is an alternating 4-cycle.
 
Assume w.l.o.g. that $y_0x_s$ is blue, then $x_sx_{s+1}$ and $y_0x_{s+2}$ are red and $x_{s+1}x_{s+2}$ is blue. 
As $C_1$ is an alternating cycle $x_{s+2i}x_{s+2i+1}$ is red for each $i\in[0,n-1]$ and $x_{s+2i+1}x_{s+2i+2}$ is blue for each $i\in[0,n-1]$. 

We may assume $y_0y_1$ is red. Since $G$ has no good pair, the next assertions follow directly from Definition \ref{definition parallel} and Lemma \ref{lemma no good pair} (ii):
\begin{itemize}
\itemsep-0.2em
\item $y_{2m-2j}x_{s-2j}$ is an exterior blue edge for each $j \in[0,l-1]$, where $l=\lcm(n,m)$, (\emph{i.e.} $y_{2j}x_{s+2j}$ is an exterior blue edge for each $j \in[0,l-1]$).
\item $y_{2m-2j-1}x_{s-2j-1}$ is an exterior red edge for each $j \in[0,l-1]$, where $l=\lcm(n,m)$, (\emph{i.e.} $y_{2j+1}x_{s+2j+1}$ is an exterior red edge for each $j \in[0,l-1]$).
\item $y_0x_{s+2}$ is red.
\item $y_{2j}x_{s+2j+2}$ is red for each $j\in[0,l-1]$.
\item $y_{2j+1}x_{s+2j+3}$ is blue for each $j\in[0,l-1]$.
\end{itemize} 

Moreover, since $y_{2j}x_{s+2j}$ is an exterior blue edge for each $j \in[0,l-1]$ we have, in particular, that $y_{2mj}x_{s+2mj}$ is an exterior blue edge for each $j\in[0,\frac{l}{m}-1]$, where $y_{2mj}=y_0$ as $2mj\equiv 0 \pmod{2m}$. 

Consider $d=\gcd(n,m)$, then $l/m=\left(\frac{nm}{d}\right)/m=n/d=n'$. Hence, $y_{0}x_{s+2mj}$ is an exterior blue edge for each $j\in[0,n'-1]$.

Observe that the sets of indices $\{s+2mj\}_{j=0}^{n'-1}$ and $\{s+2dj\}_{j=0}^{n'-1}$ are equal modulo $2n$ and thus $y_{0}x_{s+2dj}$ is an exterior blue edge for each $j\in[0,n'-1]$.

Then, again by Definition \ref{definition parallel} and Lemma \ref{lemma no good pair} (ii), we have that $y_{2i}x_{s+2dj+2i}$ is an exterior blue edge for each $i\in[0,l-1]$ and each $j\in[0,n'-1]$ and $y_{2i+1}x_{s+2dj+2i+1}$ is an exterior red edge for each $i\in[0,l-1]$ and each $j\in[0,n'-1]$.

Let $r\in [1,n'-1]$, $h\in[0,d-1]$ and $t\in[0,l-1]$. Consider the cycle 
$\xi(r,h,t)=y_{2t}x_{s+2t}x_{s+2t+1}\cdots x_{s+2t+2dr+h} y_{2t+h}y_{2t+h-1}\cdots y_{2t}$. 

\begin{claim}
$\xi(r,h,t)$ is alternating for each $r\in [1,n'-1]$, each $h\in[0,d-1]$ and each $t\in[0,l-1]$.

We know that the edge $y_{2t}x_{s+2t}$ is blue. As $x_{s+2i}x_{s+2i+1}$ is red for each $i\in[0,n-1]$ and $x_{s+2i+1}x_{s+2i+2}$ is blue for each $i\in[0,n-1]$, it follows that the alternating path $x_{s+2t}x_{s+2t+1}\cdots x_{s+2t+2dr+h}$ starts at a red edge and, depending on the parity of $h$, it ends at a blue edge whenever $h$ is even and it ends at a red edge whenever $h$ is odd.  

Moreover, since $y_{2j}x_{s+2j+2}$ is red for each $j\in[0,l-1]$ and $y_{2j+1}x_{s+2j+3}$ is blue for each $j\in[0,l-1]$; we have that $x_{s+2t+2dr+h} y_{2t+h}$ is red whenever $h$ is even and it is blue whenever $h$ is odd.

Observe that, as $y_0y_1$ is red, the edges in $C_2$ satisfy $y_{2i}y_{2i+1}$ is red for each $i\in[0,m-1]$ and $y_{2i+1}y_{2i+2}$ is blue for each $i\in[0,m-1]$. Hence, the alternating path $y_{2t+h}y_{2t+h-1}\cdots y_{2t}$, ends at a red edge and starts at a blue edge if $h$ is even and at a red edge it $h$ is odd.

Therefore, $\xi(r,h,t)$ is alternating for each $r\in [1,n'-1]$, each $h\in[0,d-1]$ and each $t\in[0,l-1]$.
\end{claim}
The cycle $\xi(r,h,t)$ has length $l(\xi(r,h,t))=1+(2dr+h)+1+h=2+2dr+2h$ for each $r\in [1,n'-1]$ and each $h\in[0,d-1]$ and thus we have alternating cycles of each even length in $[2+2d,\ 2+2d(n'-1)+2(d-1)]=[2+2d,\ 2n]$ passing through $y_{2t}$ and $x_{s+2t}$, for each $t\in[0,l-1]$.

Now, take $r\in [1,n'-1]$, $h\in[0,d-1]$ and $t\in[0,l-1]$ and consider the cycle 
$\xi(r,h,t+1)=y_{2t+1}x_{s+2t+1}x_{s+2t+2}\cdots x_{s+2t+1+2dr+h} y_{2t+1+h}y_{2t+h}\cdots y_{2t+1}$.

In a similar way it can be proved that $\xi(r,h,t+1)$ is an alternating cycle of length $2+2dr+2h$ passing through $y_{2t+1}$ and $x_{s+2t+1}$, for each $t\in[0,l-1]$.

Therefore, for each vertex $w\in V(G)$ and each even length $l$ in $[2+2d,\ 2n]$, there exists  an alternating cycle in $G$ of length $l$ passing through $w$. 

Given that $d=\gcd(n,m)$ we have that $2+2d\leq 4m$, and thus we have the result.
\end{proof}

As a consequence of Propositions \ref{proposition non-singular vertices and no good pair then almost  vertex pancyclic} and \ref{proposition non-singular vertices and no good pair then almost  vertex pancyclic 2} we have the next result:

\begin{theorem}
\label{theorem non-singular vertices and no good pair then vertex pancyclic 2 graphs}
Let $G_1$ and $G_2$ be two vertex disjoint graphs with Hamiltonian alternating cycles, $C_1=x_0x_1\cdots x_{2n-1}x_0$ and $C_2=y_0y_1\cdots y_{2m-1}y_0$, respectively, and $G\in G_1\oplus G_2$. If there is no good pair in $G$, and for each $i\in \{1,2\}$, in $C_i$  there is a non-singular vertex with respect to $C_{3-i}$. Then $G$ is vertex alternating-pancyclic.
\end{theorem}


Consider the following construction: Given a bipartite complete 2-edge-colored graph $G$, with partition $(X, Y)$, he constructs a complete 2-edge-colored graph $H$ from $G$ by adding all edges between vertices in $X$ with red color and all edges between vertices in $Y$ with blue color. This is, the induced subgraph of $H$ by $X$ is a complete red monochromatic graph and the induced subgraph of $H$ by $Y$ is a complete blue monochromatic graph. In this way, $H$ is a complete 2-edge-colored graph such that every alternating cycle in $H$ is an alternating cycle in $G$, as no alternating cycle in $H$ contains edges in $H\langle X \rangle$ or $H\langle Y \rangle$; and thus, $H$ is (vertex) alternating-pancyclic if and only if $G$ is (vertex) alternating-pancyclic. This construction is due to Das \cite{Das1982105} and later by H\"{a}ggkvist and Manoussakis \cite{Haggkvist198933}, it was used to study Hamiltonian alternating cycles in bipartite complete 2-edge-colored graphs and it is known as DHM-construction \cite{Bang-Jensen2009}.

Now we will give an example which shows that we are unable to use a similar construction to prove our results.
Consider the graph in Figure \ref{Fig-sum-of-two-4-cycles-without-possible-DHM-construction}. A DHM type construction would add edges to this graph until a complete graph is obtained, the goal of this construction is that the complete graph and the original graph have the same set of alternating cycles. However, it is impossible to do that in this graph. Observe that the graph in Figure \ref{Fig-sum-of-two-4-cycles-without-possible-DHM-construction} satisfies the hypothesis of Theorem \ref{theorem non-singular vertices and no good pair then vertex pancyclic 2 graphs} and thus it is vertex alternating-pancyclic.
\begin{figure}
\begin{center}
\includegraphics[scale=0.9]{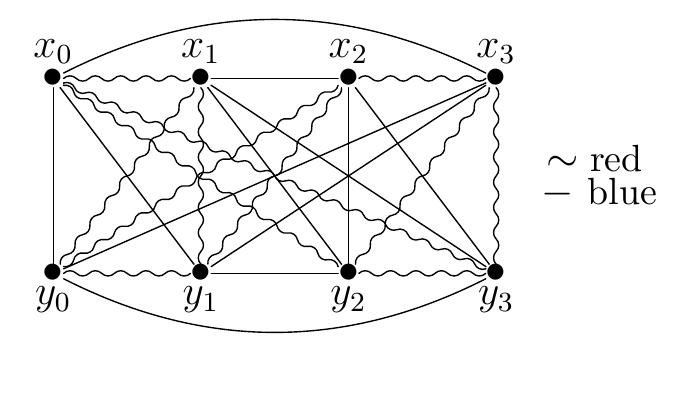}
\caption{Graph $G\in C_1 \oplus C_2$}
\label{Fig-sum-of-two-4-cycles-without-possible-DHM-construction}
\end{center}    
\end{figure}
This simple example shows that the assertions in this paper work for different graphs from bipartite complete and complete 2-edge-colored graphs.

\section{Main result}
\label{sec:main results}

In this section we will generalize Theorem \ref{theorem non-singular vertices and no good pair then vertex pancyclic 2 graphs} for a graph $G$ in the c.g.s. of $k$ vertex disjoint Hamiltonian alternating graphs.
But first we will define a \emph{good cycle}.

From the definition of good pair of edges and Remarks \ref{remark:monochromatic cycle} and \ref{remark par bueno} we obtain the following remark.

\begin{remark}
\label{remark:good pair/monochromatic 4-cycle}
Let $G$ be a graph and let $C_1=x_0 x_1\cdots x_{2n-1} x_0$ and $C_2=y_0 y_1 \cdots y_{2m-1} y_0$ be two vertex disjoint alternating cycles in $G$. If $x_s y_t$ and $y_{t'} x_{s'}$ is a good pair of edges. Then $\mathcal{C}=x_s x_{s'} y_{t'}y_t  x_s$ is a monochromatic 4-cycle such that its edges are alternatively in $E_\oplus$ and $E(C_1)\cup E(C_2)$, namely $x_{s'}y_{t'},\ y_{t}x_{s} \in E_\oplus$,  $x_sx_{s'}\in E(C_1)$ and $y_{t'}y_t\in E(C_2)$.
\end{remark}

\begin{definition}
\label{def:good monochromatic cycle}
Let $G_1$, $G_2$, \ldots, $G_k$ be a collection of pairwise vertex disjoint 2-edge-colored graphs, $G\in \oplus_{i=1}^k G_i$. 
A monochromatic 4-cycle $\mathcal{C}= v_0 v_1 v_2 v_{3}v_0$ in $G$ will be called a \emph{good cycle} when either $v_0v_1,\ v_2v_3 \in E_\oplus$ or $v_1v_2,\ v_3v_0 \in E_\oplus$, or both. This is, when two opposite edges in $\mathcal{C}$ are exterior.
\end{definition}

Notice that all proofs up to this point have been algorithmic, this is, we construct cycles step by step. 
Proceeding by induction on $k$ it is easy to prove the next theorem which is the main result of this paper.

\begin{theorem}
\label{theorem non-singular vertices and no good monochromatic cycle then vertex pancyclic}
Let $G_1$, $G_2$, \ldots, $G_k$ be a collection of $k\geq 2$ vertex disjoint graphs with Hamiltonian alternating cycles, $C_1$, $C_2$, \ldots, $C_k$, respectively, and $G\in \oplus_{i=1}^k G_i$. 
If there is no good cycle in $G$ and, for each pair of different indices $i,j\in [1,k]$, in $C_i$ there is a non-singular vertex with respect to $C_{j}$. Then $G$ is vertex alternating-pancyclic.
\end{theorem}

In this paper we considered colored generalized sums of Hamiltonian alternating graphs which contain no good pair (or its generalization, good monochromatic 4-cycle) and obtained conditions that are easy to check to determine if a c.g.s. is vertex alternating-pancyclic. The study of the case when good pairs appear will be treated in a forthcoming paper.


%
%


\bibliography{Vertex-Alternating-Pancyclism-in-2-Edge-Colored-Graphs}

\begin{thebibliography}{19}
\expandafter\ifx\csname natexlab\endcsname\relax\def\natexlab#1{#1}\fi
\providecommand{\url}[1]{\texttt{#1}}
\providecommand{\href}[2]{#2}
\providecommand{\path}[1]{#1}
\providecommand{\DOIprefix}{doi:}
\providecommand{\ArXivprefix}{arXiv:}
\providecommand{\URLprefix}{URL: }
\providecommand{\Pubmedprefix}{pmid:}
\providecommand{\doi}[1]{\href{http://dx.doi.org/#1}{\path{#1}}}
\providecommand{\Pubmed}[1]{\href{pmid:#1}{\path{#1}}}
\providecommand{\bibinfo}[2]{#2}
\ifx\xfnm\relax \def\xfnm[#1]{\unskip,\space#1}\fi
\bibitem[{Bang-Jensen and Gutin(1997)}]{Bang-Jensen199739}
\bibinfo{author}{Bang-Jensen, J.}, \bibinfo{author}{Gutin, G.},
  \bibinfo{year}{1997}.
\newblock \bibinfo{title}{Alternating cycles and paths in edge-coloured
  multigraphs: A survey}.
\newblock \bibinfo{journal}{Discrete Math.} \bibinfo{volume}{165/166},
  \bibinfo{pages}{39--–60}.
\newblock \bibinfo{note}{Graphs and combinatorics (Marseille, 1995)}.
\bibitem[{Bang-Jensen and Gutin(2009)}]{Bang-Jensen2009}
\bibinfo{author}{Bang-Jensen, J.}, \bibinfo{author}{Gutin, G.},
  \bibinfo{year}{2009}.
\newblock \bibinfo{title}{Digraphs: Theory, Algorithms and Applications}.
\newblock \bibinfo{edition}{2} ed., \bibinfo{publisher}{Springer Monographs in
  Mathematics}, \bibinfo{address}{Springer-Verlag London, Ltd, London}.
\bibitem[{Chow et~al.(1994)Chow, Manoussakis, Megalakaki, Spyratos and
  Tuza}]{Chow199449}
\bibinfo{author}{Chow, W.}, \bibinfo{author}{Manoussakis, Y.},
  \bibinfo{author}{Megalakaki, O.}, \bibinfo{author}{Spyratos, M.},
  \bibinfo{author}{Tuza, Z.}, \bibinfo{year}{1994}.
\newblock \bibinfo{title}{Paths through fixed vertices in edge-colored graphs}.
\newblock \bibinfo{journal}{Math. Inform. Sci. Hum.} \bibinfo{volume}{127},
  \bibinfo{pages}{49--–58}.
\bibitem[{Contreras-Balbuena et~al.(2017)Contreras-Balbuena, Galeana-S\'anchez
  and Goldfeder}]{Contreras-Balbuena201755}
\bibinfo{author}{Contreras-Balbuena, A.}, \bibinfo{author}{Galeana-S\'anchez,
  H.}, \bibinfo{author}{Goldfeder, I.}, \bibinfo{year}{2017}.
\newblock \bibinfo{title}{A new sufficient condition for the existence of
  alternating hamiltonian cycles in 2-edge-colored multigraphs}.
\newblock \bibinfo{journal}{Discrete Appl. Math.} \bibinfo{volume}{229},
  \bibinfo{pages}{55--63}.
\bibitem[{Das(1982)}]{Das1982105}
\bibinfo{author}{Das, P.}, \bibinfo{year}{1982}.
\newblock \bibinfo{title}{Pan-alternating cyclic edge-partitioned graphs}.
\newblock \bibinfo{journal}{Ars Combin.} \bibinfo{volume}{14},
  \bibinfo{pages}{105--114}.
\bibitem[{Dorninger(1987)}]{Doringer198795}
\bibinfo{author}{Dorninger, D.}, \bibinfo{year}{1987}.
\newblock \bibinfo{title}{On permutations of chromosomes}, in:
  \bibinfo{booktitle}{Contributions to general algebra 5: Proceedings of the
  Salzburg Conference, May 29-June 1, 1986},
  \bibinfo{publisher}{H{\"o}lder-Pichler-Tempsky, Vienna}. pp.
  \bibinfo{pages}{95--–103}.
\bibitem[{Dorninger(1994)}]{Dorninger1994159}
\bibinfo{author}{Dorninger, D.}, \bibinfo{year}{1994}.
\newblock \bibinfo{title}{Hamiltonian circuits determining the order of
  chromosomes}.
\newblock \bibinfo{journal}{Discrete Appl. Math.} \bibinfo{volume}{50},
  \bibinfo{pages}{159--–168}.
\bibitem[{Dorninger and Timischl(1987)}]{Dorninger1987321}
\bibinfo{author}{Dorninger, D.}, \bibinfo{author}{Timischl, W.},
  \bibinfo{year}{1987}.
\newblock \bibinfo{title}{Geometrical constraints on bennet’s predictions of
  chromosome order}.
\newblock \bibinfo{journal}{Heredity} \bibinfo{volume}{58},
  \bibinfo{pages}{321--–325}.
\bibitem[{Fujita and Magnant(2011)}]{Fujita20111391}
\bibinfo{author}{Fujita, S.}, \bibinfo{author}{Magnant, C.},
  \bibinfo{year}{2011}.
\newblock \bibinfo{title}{Properly colored paths and cycles}.
\newblock \bibinfo{journal}{Discrete Appl. Math.} \bibinfo{volume}{159},
  \bibinfo{pages}{1391--–1397}.
\bibitem[{Gorbenko and Popov(2012)}]{Gorbenko2012204}
\bibinfo{author}{Gorbenko, A.}, \bibinfo{author}{Popov, V.},
  \bibinfo{year}{2012}.
\newblock \bibinfo{title}{The hamiltonian alternating path problem}.
\newblock \bibinfo{journal}{IAENG Int. J. Appl. Math.} \bibinfo{volume}{42},
  \bibinfo{pages}{204--213}.
\bibitem[{Gourvès et~al.(2010)Gourvès, Lyra, Martinhon and
  Monnot}]{Gourves20101404}
\bibinfo{author}{Gourvès, L.}, \bibinfo{author}{Lyra, A.},
  \bibinfo{author}{Martinhon, C.}, \bibinfo{author}{Monnot, J.},
  \bibinfo{year}{2010}.
\newblock \bibinfo{title}{The minimum reload s-t path, trail and walk
  problems}.
\newblock \bibinfo{journal}{Discrete Appl. Math.} \bibinfo{volume}{158},
  \bibinfo{pages}{1404--–1417}.
\bibitem[{H{\"a}ggkvist and Manoussakis(1989)}]{Haggkvist198933}
\bibinfo{author}{H{\"a}ggkvist, R.}, \bibinfo{author}{Manoussakis, Y.},
  \bibinfo{year}{1989}.
\newblock \bibinfo{title}{Cycles and paths in bipartite tournaments with
  spanning configurations}.
\newblock \bibinfo{journal}{Combinatorica} \bibinfo{volume}{9},
  \bibinfo{pages}{33--–38}.
\bibitem[{Petersen(1891)}]{Petersen1891193}
\bibinfo{author}{Petersen, J.}, \bibinfo{year}{1891}.
\newblock \bibinfo{title}{Die theorie der regulären graphs}.
\newblock \bibinfo{journal}{Acta Math.} \bibinfo{volume}{15},
  \bibinfo{pages}{193--–220}.
\bibitem[{Pevzner(1995)}]{Pevzner199577}
\bibinfo{author}{Pevzner, P.}, \bibinfo{year}{1995}.
\newblock \bibinfo{title}{Dna physical mapping and properly edge-colored
  eulerian cycles in colored graphs}.
\newblock \bibinfo{journal}{Algorithmica} \bibinfo{volume}{13},
  \bibinfo{pages}{77--–105}.
\bibitem[{Wang and Li(2009)}]{Wang20094349}
\bibinfo{author}{Wang, G.}, \bibinfo{author}{Li, H.}, \bibinfo{year}{2009}.
\newblock \bibinfo{title}{Color degree and alternating cycles in edge-colored
  graphs}.
\newblock \bibinfo{journal}{Discrete Math.} \bibinfo{volume}{309},
  \bibinfo{pages}{4349--–4354}.
\bibitem[{Wirth and Steffan(2001)}]{Wirth200173}
\bibinfo{author}{Wirth, H.}, \bibinfo{author}{Steffan, J.},
  \bibinfo{year}{2001}.
\newblock \bibinfo{title}{Reload cost problems: minimum diameter spanning
  tree}.
\newblock \bibinfo{journal}{Discrete Appl. Math.} \bibinfo{volume}{113},
  \bibinfo{pages}{73--–85}.
\bibitem[{Xu et~al.(2010)Xu, Kilgour, Hipel and Kemkes}]{Xu2010318}
\bibinfo{author}{Xu, H.}, \bibinfo{author}{Kilgour, D.},
  \bibinfo{author}{Hipel, K.}, \bibinfo{author}{Kemkes, G.},
  \bibinfo{year}{2010}.
\newblock \bibinfo{title}{Using matrices to link conflict evolution and
  resolution in a graph model}.
\newblock \bibinfo{journal}{European J. Oper. Res.} \bibinfo{volume}{207},
  \bibinfo{pages}{318--–329}.
\bibitem[{Xu et~al.(2009a)Xu, Li, Hipel and Kilgour}]{Xu2009470}
\bibinfo{author}{Xu, H.}, \bibinfo{author}{Li, K.}, \bibinfo{author}{Hipel,
  K.}, \bibinfo{author}{Kilgour, D.}, \bibinfo{year}{2009}a.
\newblock \bibinfo{title}{A matrix approach to status quo analysis in the graph
  model for conflict resolution}.
\newblock \bibinfo{journal}{Appl. Math. Comput.} \bibinfo{volume}{212},
  \bibinfo{pages}{470--–480}.
\bibitem[{Xu et~al.(2009b)Xu, Li, Kilgour and Hipel}]{Xu2009353}
\bibinfo{author}{Xu, H.}, \bibinfo{author}{Li, K.}, \bibinfo{author}{Kilgour,
  D.}, \bibinfo{author}{Hipel, K.}, \bibinfo{year}{2009}b.
\newblock \bibinfo{title}{A matrix-based approach to searching colored paths in
  a weighted colored multidigraph}.
\newblock \bibinfo{journal}{Appl. Math. Comput.} \bibinfo{volume}{215},
  \bibinfo{pages}{353--–366}.

\end{thebibliography}


\end{document}